\documentclass[letterpaper,11pt]{amsart}

\usepackage{amssymb,amsbsy,amsmath,amsfonts,amssymb,amscd}
\usepackage{latexsym}
\usepackage{graphics}
\usepackage{color}
\usepackage{comment}
\input xy
\xyoption{all}

\theoremstyle{plain}
\newtheorem{thm}{Theorem}[section]
\newtheorem{theorem}[thm]{Theorem}

\newtheorem{lemma}[thm]{Lemma}

\newtheorem{proposition}[thm]{Proposition}
\theoremstyle{definition}
\newtheorem{remark}[thm]{Remark}

\newtheorem{defin}[thm]{Definition}

\newtheorem{example}[thm]{Example}

\newtheorem{question}[thm]{Question}

\numberwithin{equation}{section}

\newcommand{\sA}{{\mathcal A}}

\newcommand{\sE}{{\mathcal E}}
\newcommand{\sF}{{\mathcal F}}

\newcommand{\sL}{{\mathcal L}}

\newcommand{\sT}{{\mathcal T}}

\newcommand{\CC}{\ensuremath{\mathbb{C}}}
\newcommand{\RR}{\ensuremath{\mathbb{R}}}
\newcommand{\ZZ}{\ensuremath{\mathbb{Z}}}
\newcommand{\QQ}{\ensuremath{\mathbb{Q}}}

\newcommand{\NN}{\ensuremath{\mathbb{N}}}
\newcommand{\hol}{\ensuremath{\mathcal{O}}}
\newenvironment{dedication}
        {\begin{quotation}\begin{center}\begin{em}}
        {\par\end{em}\end{center}\end{quotation}}

\newcommand\la{\lambda}

\newcommand\al{\alpha}
\newcommand\be{\beta}
\newcommand\Ga{\Gamma}

\newcommand\ga{\gamma}

\newcommand{\Lam}{\Lambda}
\DeclareMathOperator{\Pic}{Pic}

\DeclareMathOperator{\Ext}{Ext}
\DeclareMathOperator{\Hom}{Hom}

\newcommand{\ra}{\ensuremath{\rightarrow}}

\def\eea{\end{eqnarray*}}
\def\bea{\begin{eqnarray*}}

\newcommand\dual{\mathrel{\raise3pt\hbox{$\underline{\mathrm{\thinspace d
\thinspace}}$}}}
\newcommand\qe{\ifhmode\unskip\nobreak\fi\quad $\Box$}       

\def\BOX{\hfill\lower.5\baselineskip\hbox{$\Box$}}

\newtheorem{theo}{Theorem}[section]

\newtheorem{remarkk}[theo]{Remark}
\newenvironment{rem}{\begin{remarkk}\rm}{\end{remarkk}}

\newenvironment{ex}{\begin{example}\rm}{\end{example}}

\usepackage{hyperref}
\title [ Manifolds with trivial  Chern classes]{ Manifolds with  trivial  Chern classes I: Hyperelliptic  Manifolds and a question by Severi}

\author{Fabrizio Catanese}
\address {Mathematisches Institut der Universit\"at Bayreuth\\
NW II,  Universit\"atsstr. 30\\
95447 Bayreuth}
\email{fabrizio.catanese@uni-bayreuth.de}
\address{  Korea Institute for Advanced Study, Hoegiro 87, Seoul, 
133--722.}
\thanks{AMS Classification: 14F, 14K, 14C25\\
 }
 
 \date{\today}

\begin{document}

\begin{abstract}

 We  give a negative answer to a question posed by Severi in 1951, whether the Abelian Varieties are the only projective manifolds with
trivial  Chern classes. 

 By Yau' s celebrated result, compact K\"ahler manifolds with  trivial Chern classes
 must be flat, that is, they belong to the class of Hyperelliptic Manifolds (quotients $T/G$  of a complex torus $T$ by the free action of a finite  group $ G$).

We exhibit simple examples of projective Hyperelliptic Manifolds 
which are not Abelian varieties and whose Chern classes are zero not only in integral cohomology, but also in the Chow ring.

We prove moreover that the Bagnera-de Franchis manifolds (quotients $T/G$  as
above but where the   group $ G$ is cyclic) have topologically trivial tangent bundle.

Our results  naturally lead to the question of classifying all compact K\"ahler manifolds with
topologically trivial tangent bundle, and all the counterexamples to Severi's question.

\end{abstract}

\maketitle

\begin{dedication}
In   memory of   Mario Baldassarri (1920-1964).
\end{dedication}

\tableofcontents

\section*{Introduction and history of the problem.}

The purpose of this article is  to give a negative answer to a question raised by Severi 
in 1951\cite{severi} \footnote{ the question was formulated in terms of the so-called canonical systems 
$K_0(X) , \dots, K_{n-1}(X)$ of a projective manifold, but these were shown in 1955 by  Nakano \cite{nakano} to be the so called  Chern classes of the cotangent bundle
 (see \cite{atiyah} for an historical  account and \cite{fulton} as a general reference)},
of which I became aware reading a paper by 
Baldassarri \cite{baldassarri}: Severi  asks whether the Abelian varieties 
can be characterized as the projective manifolds whose   Chern classes are all trivial.

For a projective complex Manifold, 
given a vector  bundle $\sE$ on $X$, we have (by Grothendieck's  method \cite{grothendieck})
 Chern classes $c_i(\sE) \in \sA^i(X)$ in the Chow ring of $X$ of cycles modulo rational equivalence
 and their respective images, 
 the integral Chern classes $c_{i, \ZZ}(\sE) \in H^{2i}(X, \ZZ)$,
  the rational  Chern classes $c_{i, \QQ}(\sE) \in H^{2i}(X, \QQ)$, 
 and the real Chern classes $c_{i, \RR}(\sE) \in H^{2i}(X, \RR)$.
 
 The latter classes make also sense if  $X$ is  a cKM = compact K\"ahler Manifold.
 
The Chern classes of $X$ are the Chern classes of the  tangent bundle $\Theta_X$, 
hence  we get a series of homomorphic images 
 $$ c_i(X ) \mapsto c_{i, \ZZ}(X) \mapsto c_{i, \QQ}(X) \mapsto c_{i, \RR}(X),$$
 where the last map is always injective.
 
 Therefore there are three questions, 
 
 \begin{itemize}
 \item
  (`Severi's question') classify all complex projective Manifolds $X$ with all  Chern classes trivial in the Chow ring;
  \item
  classify all cKM's  $X$ with trivial integral Chern classes;
  \item
  classify all cKM's $X$ with trivial rational Chern classes.
  \end{itemize}
  And only the third question is fully answered, since 1978 \cite{yau}.
  
  \bigskip
  
As we said,  Severi's question has a negative answer.

 If we assume that the all the Chern classes are zero in integral homology
 ($c_i(X) = 0 \in H^{2i} (X, \ZZ),  \forall i$) we shall see that  an  example is given already in dimension 2 by the Hyperelliptic surfaces.

If we make the stronger assumptions that the Chern classes are  zero in the Chow ring of $X$, then 
the counterexamples
start in dimension 3, since for hyperelliptic surfaces 
 $ c_1(X) \neq 0 \in Pic(X)$. 
 
 \medskip
 
 This is our full result:
 
 \begin{theorem}
(a) The tangent bundle of a Bagnera-de Franchis manifold $X = T/G$ ($X$ is the quotient of a complex torus $T$ by a cyclic group $G$ acting freely and containing no translations)
is topologically trivial, in particular  all its integral
Chern classes $c_i(X) =0 \in H^*(X,\ZZ)$.

(b) There are projective Bagnera-de Franchis manifolds $X = T/G$, which are not complex tori,
such that all its 
Chern classes $c_i(X) $ are  zero in the Chow ring of $X$.

(c) There are some Hyperelliptic manifolds  $X = T/G$ ( these are the compact K\"ahler manifolds with trivial Chern classes in rational cohomology) 
such that not all their  integral 
Chern classes $c_i(X) \in H^*(X,\ZZ)$ are equal to zero.

\end{theorem}
 
  \bigskip
  
  To clarify the history of the problem, recall that the characterization of the Manifolds with all Chern classes zero in rational (equivalently, in real) cohomology 
  was solved  in 1978, thanks to Yau's celebrated theorem \cite{yau} 
  about  the existence of 
 K\"ahler-Einstein metrics on manifolds with $c_1 (X) = 0 \in H^2(X, \QQ)$.
 
From this result, as explained in Kobayashi's book, page 116 of \cite{kobayashi},  follows that the 
compact  K\"ahler manifolds   (cKM) with $c_1(X) = c_2(X)= 0 \in H^*(X, \QQ)$
are the Hyperelliptic manifolds, the quotients of a complex torus by the free action of a finite group $G$.
Indeed, once one knows that we have a  K\"ahler-Einstein metric, that the manifold is flat had been proven by Apte
\cite{apte}
in the 50's. 

Hence, after 1978, Severi's question became  a question concerning Hyperelliptic manifolds.

The technical core of this article is the investigation of the Picard group
of line bundles on Hyperelliptic Manifolds, via some Grothendieck spectral sequences:  the investigation becomes easier 
when the group $G$ is cyclic, because of the vanishing of the group of Schur multipliers
and in view of several other very special features.

Then we are able to use the fact that, if the group $G$ is Abelian, then the tangent bundle 
of $X$ is a direct sum of line bundles: this  is very convenient  since  a line bundle
is topologically trivial if and only if its integral Chern class is zero.

The analysis of the general case where $G$ is not Abelian seems challenging.

\bigskip

Our above  theorem shows   that the situation for a general Hyperelliptic Manifold 
 is not fully clear: there are examples with  Chern classes trivial also in integral cohomology, or even in the Chow ring 
of rational equivalence classes,
but there are also examples with nontrivial  Chern classes  in integral cohomology. 

Hence our theorem raises the 
interesting problem of 
a  complete classification
of the  (hyperelliptic) manifolds with  $c_i(X) = 0 \in  H^*(X, \ZZ) \   \forall i$,
or with topologically trivial tangent bundle, or (when they are algebraic) with trivial Chern classes
in the Chow ring.

\section{Hyperelliptic manifolds and varieties}

We recall here some basic facts about the theory of Hyperelliptic manifolds,
starting from their definition.

\medskip

The French school   of Appell, Humbert, Picard, Poincar\'e   defined the Hyperelliptic Varieties as those smooth projective varieties whose universal covering is
biholomorphic to $\CC^n$ (in particular the Abelian varieties are in this class). For $n=1$ these are just the elliptic curves, whereas  the Hyperelliptic varieties of dimension  $2$ were classified by 
 Enriques  and Severi  (\cite{es}) and by  Bagnera and De Franchis (\cite{bdf}): both pairs were awarded the prestigious Bordin Prize for this achievement.
 
 Kodaira \cite{kod} showed instead that if we
take the wider class of compact complex manifolds of dimension $2$ whose universal covering is
$\CC^2$, then there are other non algebraic and non K\"ahler surfaces, called nowadays Kodaira surfaces (beware: these are not the so-called Kodaira fibred surfaces!).

Iitaka conjectured that if a compact  K\"ahler Manifold $X$ has universal covering 
biholomorphic to $\CC^n$, then necessarily  $X$ is a quotient $ X = T / G$ of a complex torus $T$ by the free action of a finite group $G$ (which we may assume to contain no translations).

The conjecture by Iitaka was proven in dimension $2$ by Kodaira, and in dimension $3$ by Campana and Zhang \cite{cz}.
Whereas it was shown in \cite{chk} that,  if the abundance conjecture holds, then a   projective smooth variety $X$ with universal covering $\CC^n$ is a Hyperelliptic variety according to the following definition.

\begin{defin}
A Hyperelliptic Manifold $X$  is defined to be a quotient $ X = T / G$ of a complex torus $T$ by the free action of a finite group $G$ which contains no translations. 

We say that $X$ is a  Hyperelliptic Variety if moreover the torus $T$ is projective, i.e., it is an Abelian variety $A$, that is, 
$A$ possesses an ample line bundle $L$.
\end{defin}

 If the group $G$ is a cyclic group $\ZZ/m$, then such a quotient is called (\cite{bcf}, \cite{topmethods}) a Bagnera-De Franchis manifold.
 
In  dimension $n=2$, a hyperelliptic manifold $X$ is necessarily projective, and $G$ is necessarily cyclic, whereas in dimension $ n\geq 3$ the only examples with $G$ non Abelian have 
$G = D_4$ and were classified in \cite{U-Y} and \cite{cd} (for us $D_4$ is the dihedral group  of order $8$).

Indeed, (see for instance \cite{ccd}) every Hyperelliptic Manifold is a deformation of a Hyperelliptic Variety, so that
a posteriori the two notions are related to each other, in particular the underlying differentiable manifolds are the same.

There are at least three important research directions concerning Hyperelliptic Varieties:

\begin{enumerate}
\item
Establish  Iitaka's conjecture.
\item
Understand and classify Hyperelliptic Manifolds.
\item
Construct interesting manifolds as submanifolds (e.g., Hypersurfaces) of Hyperelliptic Manifolds.

\end{enumerate}

Question (1) is essentially a question about fundamental groups of  compact K\"ahler Manifolds: since (cf. for instance \cite{topmethods} Coroll. 82, page 356)
any compact K\"ahler Manifold $X$ with contractible universal cover and with $\pi_1(X)$ Abelian is a complex torus.
Hence the main point is to show that if a compact  K\"ahler Manifold $X$ has universal covering 
biholomorphic to $\CC^n$, then necessarily  $\pi_1(X)$ has an Abelian subgroup of finite index.

More generally one can ask: 

\begin{question}
Given a compact complex manifold with universal covering $\tilde{X} \cong \CC^n$,
is the fundamental group $\pi_1(X)$ a solvable group?
\end{question}

Question (2) above  is  instead essentially a difficult algebraic question: since  if $X$ is a Hyperelliptic Manifold,
and $\Ga : = \pi_1(X)$, then we have an exact  
sequence of groups  
$$ (*) \ \ 0 \ra \Lam \ra \Ga \ra G \ra 1 ,$$
where $\Lam = \pi_1(T) \cong \ZZ^{2n}.$

This leads  (see \cite{Catanese-Corvaja}) to the following definition of an 
 {\bf abstract torsion free even Euclidean cristallographic group}.

\begin{defin}\label{cristall}
(i) We  say that a group $\Ga$ is an abstract Euclidean cristallographic group if there exists an exact
sequence of groups
$$ (*) \ \ 0 \ra \Lam \ra \Ga \ra G \ra 1 $$
such that
\begin{enumerate}
\item
$G$ is a  finite group
\item
$\Lam$ is free abelian (we shall denote  its rank by $r$)
\item
Inner conjugation $ Ad : \Ga \ra  Aut (\Lam) $ has Kernel exactly $\Lam$,
hence $ Ad$  induces an embedding, called {\bf Linear part},
$$ L : G \ra GL (\Lam) : = Aut (\Lam) $$
(thus $L(g) (\la) =  Ad (\ga) (\la) = \ga \la \ga^{-1}, \ \forall \ga {\rm \ a \ lift \ of }\  g$) 

\end{enumerate}

(ii) A cristallographic group $\Ga$ is said to be {\bf even} if:
\begin{itemize}
\item
(ii.1)  $\Lam$ is a free abelian group  of even rank $ r = 2n$
\item
(ii.2) $\Ga$ is {\bf$G-$even}, equivalently,  there exists a Hodge decomposition   $$\Lam \otimes_{\ZZ} \CC = H^{1,0} \oplus \overline{H^{1,0} }$$ 
which
 is invariant for the $G$-action (i.e., $H^{1,0}$ is a $G$-invariant subspace); this is equivalent to:
 \item 
 (ii.2 bis) considering the associated   faithful representation $ G \ra Aut (\Lam)$,  for each real representation $\chi$ of $G$, 
the $\chi$-isotypical component  {\bf $$ M_{ \chi} \subset \Lam \otimes_{\ZZ} \RR$$ has even dimension } (over $\CC$).
 \end{itemize}
 
 (iii) $\Ga$ is said to be torsion-free if there are no elements of finite order inside $\Ga$.

(iv) An {\bf affine realization defined over a field $ K \supset \ZZ$ }  of an abstract Euclidean cristallographic group $\Ga$  is a homomorphism
(necessarily injective) 
$$\rho : \Ga \ra Aff (\Lam \otimes_{\ZZ} K)$$ such that 

[1] $\Lam$ acts by translations on $ V_K := \Lam \otimes_{\ZZ} K$,  $ \rho(\la) (v) =   v + \la$,

[2]  for any $\ga$ a lift of $g \in G$ we have:
$$  V_K \ni v \mapsto \rho(\ga) (v) =  Ad (\ga) v + u_{\ga} = L(g) v + u_{\ga}, \ {\rm for \ some } \ \ u_{\ga} \in  V_K.$$ 

(v) More generally we can say that an affine realization of $\Ga$  is obtained via a lattice $\Lam ' \subset \Lam \otimes_{\ZZ} \QQ$
if there exists a homomorphism $\rho ' : \Ga \ra Aff (\Lam ')$ such that $\rho = \rho ' \otimes_{\ZZ}  K$ (then necessarily $\Lam \subset \Lam'$).

\end{defin}

Extending previous classical results of Bieberbach \cite{bieb1, bieb2}, in \cite{Catanese-Corvaja} was proven:

   \begin{thm}\label{affinereal}
Given an abstract Euclidean cristallographic group there is one and only one  class of affine realization, for each    field $K \supset \ZZ$.

There is moreover an effectively computable  minimal number $d \in \NN$ such that the  realization is obtained via $\frac{1}{d} \Lam$.
\end{thm}

 The above  theorem of  \cite{Catanese-Corvaja} says in particular  that conversely, given such a torsion free  even Euclidean cristallographic 
 group $\Ga$, there are Hyperelliptic Manifolds  with $\pi_1(X) \cong \Ga$.
 
  Moreover the Hyperelliptic Manifolds are the compact K\"ahler Manifolds which are  $K(\Ga, 1)$ 's  for abstract torsion-free even Euclidean cristallographic groups $\Ga$ (recall that a $K(\Ga, 1)$ is a space $X$ with contractible universal covering and with 
$\pi_1(X) \cong \Ga$).

 Euclidean cristallographic groups were investigated by Bieberbach (\cite{bieb1, bieb2}) who proved that, in each dimension, there 
 is a finite set of isomorphism classes
 (the proof uses Minkowski's geometry of numbers, but is to our knowledge not effective and does not lead to a classification). 
 
 We end this section with an observation, on the automorphism group of Hyperelliptic Manifolds, which will
 be quite important in the sequel (part II).
 
 \begin{proposition}\label{hyp}
 Let $X = T/G$ be a Hyperelliptic Manifold. 
 
 Then its group of Automorphisms is the quotient
 $ Aut(X) = Aut(T)^G / G$, where  
  $$   Aut(T)^G  : = N_{Aut(T)}(G) \subset Aut(T)$$
   is the Normalizer of $G$ in $Aut(T)$.
   
   In particular the connected component of the Identity $ Aut^0(X)$ is isomorphic to
   the subtorus $$ T' \subset T, \ T' = \{ x | g(x) = x , \forall g \in G\}.$$
   
   The group $ Aut^0(X)$ may then be trivial if $n : = dim (X) \geq 3$ and $G$ is not cyclic.
   
   Equivalently,  $ Aut^0(X)$ is trivial if and only if $H^1(X, \hol_X)=0$.
 \end{proposition}
 
\begin{proof}
The subgroup $\Lam$ is a characteristic sugbgroup of $\Ga$, hence, for each $\Phi \in Aut(X)$, 
$\Phi$ lifts to an automorphism $\phi \in Aut(T)$, which normalizes $G$.

Then the linear part of $\phi$ defines a homomorphism of $Aut(X) \ra GL(\Lam)$,
and, since the image is discrete, $ Aut^0(X)$ consists of translations.
A translation $ z \mapsto z + b $ normalizes $G$ if and only if $G(b) = b$.

Finally, writing $ T = V/\Lam$,  the subspace $V^G: =  \{ v | Gv = v\}$ is not  trivial
if $G$ is cyclic, since $G$ acts freely, but for $ n \geq 3$ we have the 
 case of $G=D_4$, where $V^G =0$ (see \cite{cd}).
 
 Now,   $V^G=0$ if and only if  $(V^{\vee})^G=0$, equivalently,  $H^1(X, \hol_X)=0$.
 
 \end{proof}
 
 \section{The Picard group of a Hyperelliptic Manifold}
 
 If $X  = T/G$ is a Hyperelliptic Manifold,  we want to analyze the exponential exact sequence:
 
 $$ 0 \ra  H^1 (X, \ZZ) \ra  H^1 (X, \hol_X) \ra   \Pic (X) = H^1 (X, \hol^*_X) \ra H^2 (X, \ZZ) \ra H^2 (X, \hol_X) $$
 which leads to:
 $$ 0 \ra   \Pic^0 (X) = H^1 (X, \hol_X) /  H^1 (X, \ZZ)   \ra   \Pic (X)  \ra  NS(X) \subset H^2 (X, \ZZ)  .$$
 
  Since $ f : T \ra X$ is finite, we have, by the Leray spectral sequence,   
  $$ ( *) \  H^i (X, \hol_X) =  H^i (T, \hol_T)^G$$ for $ i \geq 1.$
  
  We write $ T = V / \Lam$, and $ V = V_1 \oplus V_2$, in such a way that the linear part $L$ of the action of $G$ acts as the identity on $V_1$,
  while $V_2$ is a direct sum of nontrivial irreducible representations of $G$.   
  Follows then from Dolbeault's theorem: 
  $$  H^1 (X, \hol_X) \cong   \overline{ V_1}^{\vee} , \  H^2  (X, \hol_X) \cong \Lam^2 ( \overline{ V}^{\vee})^G =  \Lam^2 ( \overline{ V_1}^{\vee}) \oplus  \Lam^2 ( \overline{ V_2}^{\vee})^G.$$
  
 Write $\Lam_i : = \Lam \cap V_i$. Since $V_1$ is defined over $\QQ$ and contains $\Lam_1$ as a lattice, 
  we have an  exact sequence
  $$ 0 \ra \Lam_1 \oplus \Lam_2 \ra \Lam \ra \Lam^* \ra 0, $$ 
  where $\Lam^*$ is a finite  group, and we obtain
  $$ (00) \  0 \ra \Hom (\Lam , \ZZ) \ra  ( \Hom (\Lam_1 , \ZZ) \oplus \Hom (\Lam_2 , \ZZ) ) \ra \Ext^1(\Lam^* , \ZZ) \ra 0.$$

  Apply  now the Grothendieck spectral sequence 
  $$ H^p (G, H^q (T, \sF)) \Rightarrow H^{p+q} (X, f_* \sF^G), $$
  first to $\sF = \ZZ$, then to $\sF = \hol^*_T$. 
  
  In the first case $\sF = \ZZ$ we  get, since $H^1 (G, \ZZ) = \Hom (G, \ZZ) = 0$, and from the spectral sequence diagram
  (here $ \Lam^{\vee}  : = \Hom (\Lam , \ZZ)$)
   \begin{equation}
 \left(\begin{matrix} H^2 (T, \ZZ) ^G & \dots & \dots & \dots
 \cr (\Lam^{\vee})^G = H^1 (T, \ZZ) ^G & H^1 (G,  \Lam^{\vee}) &H^2 (G,  \Lam^{\vee}) &H^3 (G,  \Lam^{\vee}) 
 \cr \ZZ& 0 &H^2( G , \ZZ)&H^3( G , \ZZ)
\end{matrix}\right).
\end{equation} 
 that
  
  $$ (**) \  0    \ra    H^1(X, \ZZ) = H^1(\Ga , \ZZ) =\Hom (\Ga , \ZZ)  \cong  Ker (\psi) ,$$
  where 
  $$ \psi :  H^1 (T, \ZZ) ^G = \Hom (\Lam , \ZZ)^G  \ra H^2( G , \ZZ).$$
  
  Also, we have a filtration on $H^2(X, \ZZ)$ with graded pieces:
  $$ H^2( G , \ZZ) / Im (\psi) , \ Ker [\varphi:  H^1 (G, \Hom (\Lam , \ZZ)) \ra  H^3( G , \ZZ)], \ $$
  $$  Ker [ Ker [ H^2(T, \ZZ)^G \ra  H^2 (G, \Hom (\Lam , \ZZ))] \ra Coker (\varphi)].$$ 

Observe now that, due to exact sequence $ 0 \ra \ZZ \ra \CC \ra \CC^* \ra 0$, and since $H^i(G, \CC) = 0$ for $i \geq 1$, 
$$H^2( G , \ZZ) \cong H^1 (G, \CC ^*) \cong \Hom (G, \CC^*) \cong \Hom (G^{ab}, \CC^*),$$
while $H^3( G , \ZZ) \cong H^2 (G, \CC ^*)$, the group which classifies the central extensions
$$ 1 \ra \CC^* \ra G' \ra G \ra 1.$$

In the second case  ($\sF = \hol^*_T$) the Grothendieck spectral sequence yields  the exact sequence

$$ 0 \ra H^1(G, \CC^*) = \Hom (G, \CC^*)  \ra H^1 (\Ga, \hol_V^*) = \Pic(X) \ra $$
$$ \ra H^1 (\Lam, \hol_V^*)^G = \Pic(T)^G \ra H^2(G, \CC^*) .$$

This exact sequence is more geometrical, it  is the standard sequence saying that G-linearized line bundles on $T$ map to $G$-invariant line bundle classes,
and two linearizations differ by a character $\chi : G \ra \CC^*$.  

The sequence gives an obstruction, for a $G$-invariant line bundle class,  to admitting a $G$-linearization,
and the obstruction takes values in $H^2(G, \CC^*)$. This obstruction group  is trivial  for instance if $G$ is a cyclic group.

The group  $H^2(G, \CC^*)$ is called  the group of Schur multipliers, and the 
Schur multiplier 
 that we obtain from the last arrow  in the above sequence
is the class of the Thetagroup of $L$: if a line bundle class $L$ on $T$ is $G$-invariant,
Mumford, \cite{mumford} pages 221 and foll., defined the Thetagroup $\Theta(L)$ as the group of the isomorphisms of $L$ with $g^*(L)$, so that we have the exact sequence
$$ 1 \ra \CC^* \ra \Theta(L) \ra G \ra 1. $$
This is a central extension, hence it   is classified by an element in $H^2(G, \CC^*)$ which measures the obstruction to splitting the above
exact sequence (that is, to lifting the action of $G$ to $L$).

\begin{ex}
Consider the canonical line bundle $K_X$ on $X$. Its pull back is the canonical line bundle $K_A$, which is a trivial line bundle
 $K_A \cong \hol_A$. Both line bundles are $G$-linearized, but the corresponding linearizations are different.
 $G$ acts trivially on $H^0(A, \hol_A)$, while it acts on $H^0(A, K_A)$ through the representation $det (L(G))$.
 
 Hence {\bf the canonical line bundle $K_X$ of  a Hyperelliptic manifold   is trivial if and only if the representation $L : G \ra GL(V)$ is unimodular
 (has determinant $=1$).}
\end{ex}

From the previous discussion it is apparent that the  main group to be investigated is then $H^1 (\Lam, \hol_V^*)^G = \Pic(T)^G$.
 
We use here then the  exact sequence for the Picard group of $T$ (derived from the exponential sequence):
$$ (***) \ 0 \ra \Pic^0(T) \ra \Pic(T) \ra NS(T) \ra 0, $$
$$\ NS(T) = Ker [ H^2(T, \ZZ) \ra H^2(T, \hol_T) ].   $$
This sequence is very explicit: by the Theorem of Appell-Humbert, $NS(T) $ is the space of Hermitian forms $H$ on $V$
whose imaginary part $E$ takes integral values on $\Lam$.

Indeed, interpreting $ \Pic(T) = H^1 (\Lam, \hol_V^*)$ we get the cocycles in Appell- Humbert normal form:
$$ f_{\la} (z) =  \rho(\la) \ \exp (  \pi ( H (z, \la) + \frac{1}{2} \pi H (\la, \la))) ,$$
where $\rho$ is a semicharacter for $E$, that is, $\rho : \Lam \ra \CC^*$ 
satisfies  $$\rho (\la + \la') = \rho (\la )\rho ( \la') \exp (\pi i E(\la, \la')).$$

In this interpretation $ \Pic^0(T) = T^* := \overline{V}^{\vee} / \Hom(\Lam, \ZZ) \cong  \Hom(\Lam, \RR)  / \Hom(\Lam, \ZZ),$
and we get a character $\chi : \Lam \ra \CC^*$ by composing with $ y \mapsto \exp (2 \pi i  y)$.

We take the exact sequence of $G$-invariants associated to $(***)$:
  
  $$ (****) \ 0 \ra (T^*)^G = \Pic^0(T)^G \ra \Pic(T)^G \ra NS(T)^G  \ra H^1(G, T^*) .   $$
  The last arrow measures the obstruction for an invariant form $H$  in $NS(T)$ to come from an invariant class in $\Pic(T)$; 
  and
  the obstruction cocycle associates to an element $g \in G$ 
  $$  g^*(L) \otimes L^{-1} \in T^* = Pic^0(T),$$ for $L$  a line bundle with Chern class $H$.

Using the Appell-Humbert theorem it is easy to calculate $NS(T)^G$: these are the Hermitian forms
$H$
as above which are $G$-invariant, hence $NS(T)^G = H^2(T, \ZZ)^G \cap H^{1,1}(T)$.

As a final remark, since $(T^*) =  \overline{V}^{\vee} / \Hom(\Lam, \ZZ)$, taking $G$-invariants we obtain:

$$ \  (Pic)  \ \ 0 \ra  (\overline{V}^{\vee})^G / \Hom(\Lam, \ZZ)^G \ra (T^*)^G \ra  H^1(G, \Hom(\Lam, \ZZ)) \ra 0,$$
and 
$$  H^1(G, T^*) \cong  H^2(G, \Hom(\Lam, \ZZ)),$$
which brings us back to the first spectral sequence.

\bigskip

 \begin{lemma}
 If $X$ is a Bagnera -De Franchis Manifold, , then 
 $$ \psi :  H^1 (T, \ZZ) ^G = \Hom (\Lam , \ZZ)^G  \ra H^2( G , \ZZ)$$ is onto.
 
 Hence in particular 
 $$ (**) \  0    \ra    H^1(X, \ZZ) \ra H^1 (T, \ZZ) ^G  \ra  H^2( G , \ZZ) \ra 0 $$
 is exact.
 \end{lemma}
 \begin{proof}
 Abelianizing the exact sequence
 $$ 0 \ra \Lam \ra \Ga \ra G \ra 1$$
 we obtain
$$ 0 \ra \Lam/ (\Lam \cap [\Ga, \Ga] )\ra \Ga^{ab}  \ra G^{ab} \ra 0.$$

In particular we have that the Kernel, $\Lam/ (\Lam \cap [\Ga, \Ga] )$, is a quotient of the
space of coinvariants $\Lam_G = \Lam/ ( [\Lam, \Ga] )$.

In the case where $G$ is cyclic, generated by the image of $\ga$, then $[\Ga, \Ga]$ equals
just $[\Lam, \Ga],$ since the brackets $[\ga^i, \ga^j]$ are trivial. Hence in this case
the Kernel $\Lam/ (\Lam \cap [\Ga, \Ga] )$ equals $\Lam_G $.

 We have (in general)  the exact sequence
$$ 0 \ra N : = \langle Im (I - L(g)) \rangle_{g \in G} \ra \Lam \ra \Lam_G \ra 0, $$ 

Hence, denoting as usual $ M^{\vee} : =  \Hom ( M, \ZZ)$, we get

$$ 0  \ra  (\Lam_G)^{\vee} \ra \Lam^{\vee} \ra N^{\vee} $$
and we have
$$ (\Lam_G)^{\vee} = (\Lam^{\vee})^G \subset  (\Lam_1)^{\vee},$$
where the first equality is by definition and the second  inclusion  holds since we have
$$ 0 \ra \Lam^{\vee} \ra (\Lam_1)^{\vee} \oplus (\Lam_2)^{\vee} \ra \Ext^1 (\Lam^*, \ZZ) \ra 0,$$
hence  $(\Lam^{\vee})^G \subset ((\Lam_1)^{\vee} \oplus (\Lam_2)^{\vee})^G = (\Lam_1)^{\vee}$.

In the BdF case, where $G$ is cyclic, starting from the exact sequence
$$ 0 \ra \Lam_G \ra \Ga^{ab}  \ra G^{ab}= G  \ra 0,$$
 we make the following
 
 {\bf CLAIM:} we have  the exact sequence
$$ (ES) \  0 \ra  H^1 (X, \ZZ) = (\Ga^{ab})^{\vee} \ra  (\Lam_G)^{\vee} =   (\Lam^{\vee})^G = H^1 (T, \ZZ) ^G \ra \Ext^1(G^{ab}, \ZZ)  \ra 0.$$

Once the above claim  is shown  we can conclude since $$\Ext^1(G^{ab}, \ZZ)  \cong G^{ab} \cong \Hom (G, \CC^*) \cong H^2( G , \ZZ).$$

To show that the above sequence (ES) is exact, we need to show the 
isomorphism
$\Ext^1(\Ga^{ab}, \ZZ) \cong \Ext^1(\Lam_G, \ZZ)$, which in turn follows if we show that we have an isomorphism 
$ Tors (\Lam_G) \cong Tors (\Ga^{ab})$.

In order to show this, 
we go back to  the description of Bagnera de Franchis manifolds,
as done in \cite{topmethods}, Proposition 16 page 309:

$ X = T /G$, with $ T = ( A_1 \times A_2)/ \Lam^*$, where 
$A_1, A_2$ are complex tori and $ \Lam^* \subset A_1 \times A_2$ is a finite subgroup,
such that 

\begin{enumerate}
\item
$ \Lam^*$ is the graph of an isomorphism between subgroups $\sT_1 \subset A_1, \sT_2 \subset A_2$,
\item
$(\al_2 - Id) \sT_2 =0$, where
\item 
$G$ is generated by $g$ such that
$$g (a_1, a_2) = (a_1 + \be_1, \al_2 (a_2)),$$
and such that the subgroup of order $m$ generated by $\be_1$ intersects $\sT_1$ only in $\{0\}$.
\item
In particular,  $ X = ( A_1 \times A_2)/ (G \times \Lam^*).$
\end{enumerate}

By property (2) it follows that $(Id - L_g) (\Lam^*) =0$, hence  $(Id - L_g) (\Lam) \subset \Lam_1 \oplus \Lam_2$
and indeed, if we define  $\Lam_2' \subset \Lam_2 \otimes \QQ$ via the property that $\Lam_2' / \Lam_2 \cong \sT_2$,
then $(Id - L_g) (\Lam) = (\al_2 - Id) \Lam_2'  \subset  \Lam_2$, since the vectors in the image
 have  first coordinate equal to zero. Then we have an exact sequence
  $$ 0 \ra  \Lam_1  \oplus [ \Lam_2 / (\al_2 - Id)(  \Lam_2' ) ] \ra \Lam_G \ra \Lam^*\ra 0 .$$

 We apply now Proposition 25, page 315 of \cite{topmethods}, stating that 
 $$Alb(X) = A_1/  ( \sT_1 \oplus \langle \be_1\rangle),$$
 hence if we write $H_1(X,\ZZ) = Tors(H_1(X,\ZZ)) \oplus H_1(X,\ZZ)_{free}$,
 $$\ H_1(X,\ZZ)_{free}/ \Lam_1 = \Ga^{ab}_{free} / \Lam_1 \cong \Lam^* \oplus (\ZZ/m).$$

 We conclude observing that  the torsion group of $\Lam_G$ contains   the finite subgroup $ [ \Lam_2 / (\al_2 - Id)(  \Lam_2' ) ] $,
 which is therefore   contained in the torsion group of $\Ga^{ab}$: the latter cannot however be larger since
 the quotient
 $$H_1(X,\ZZ) / (\Lam_1\oplus [ \Lam_2 / (\al_2 - Id)(  \Lam_2' ) ] ) \cong  \Lam^* \oplus (\ZZ/m) =  H_1(X,\ZZ)_{free}/ \Lam_1 .$$

 \end{proof}
 
  \begin{remark}
 In general the surjective map $\Lam_G  \ra \Lam/ (\Lam \cap [\Ga, \Ga] )$ is not injective:
 for instance, in the case of the Hyperelliptic threefold with $G = D_4$, it has a kernel $\cong \ZZ/2$.
 
 \end{remark}
 
  \begin{lemma}
 If $X$ is a Bagnera -De Franchis Manifold, , then 
 we have an exact sequence 
  $$ (**) \  0    \ra    H^1(G, (\Lam^{\vee})) \ra H^2 (X, \ZZ)   \ra   ker [ H^2 (T, \ZZ)^G \ra  H^2( G , (\Lam^{\vee}))]  \ra 0 .$$
  
  In particular, the torsion group of $H^2 (X, \ZZ) $ is the  group  $H^1(G, (\Lam^{\vee})),$
  which is an $m$-torsion group.

 \end{lemma}
 
\begin{proof}
The first assertion  follows from the first spectral sequence, since $H^3(G, \ZZ) = H^2 (G, \CC^*) = 0$.

For the second assertion, we notice that the third term in the exact sequence is contained in $H^2 (T, \ZZ)$,
hence it is torsion free.

Observe moreover  that, $G$ being cyclic and generated by $g$, a  cocycle in $H^1(G, (\Lam^{\vee}))$  is fully determined
by the element $ f(g) \in \Lam^{\vee},$ actually $f(g) \in Ker ( 1 + g + g^2 \dots + g^{m-1})$. 

Whereas the coboundaries 
are the elements inside  $ Im (1-g)$. 

Hence $x : = f(g)$ is cohomologous to $ g x$, which is cohomologous
to $g^2 x$, and  proceeding in this way we infer that $ m x$ is cohomologous to zero, as we wanted to show. 

\end{proof}

We summarize our results in the following Theorem, part (1') therein is due to Andreas Demleitner.

\begin{theorem}\label{summarize}
	Let $X = T/G$ be a hyperelliptic manifold. The following statements hold:
	\begin{itemize}
		\item[(1)] The sequence of $G$-linearized line bundles on $X$ is 
		$$ 0 \ra H^1(G, \CC^*) = \Hom (G, \CC^*)  \ra H^1 (\Ga, \hol_V^*) = \Pic(X) \ra $$
		$$ \ra H^1 (\Lam, \hol_V^*)^G = \Pic(T)^G \ra H^2(G, \CC^*),$$
		and a line bundle $L \in \Pic(T)^G$ admits a linearization if and only if its class  maps to zero in $H^2(G,\CC^*)$. Moreover, two linearizations on $L$ differ by a character $\chi \colon G \to \CC^*$. 
		\item[(2)] $H^1(X,\ZZ)$ sits in an exact sequence
		$$0 \to H^1(X,\ZZ) \to H^1(T,\ZZ)^G \to Im(\psi) \to 0, $$
		$$\psi \colon H^1(T,\ZZ)^G \to H^2(X,\ZZ).$$
	\end{itemize}
	If $X = T/G$, $T = (V_1 \oplus V_2)/\Lam$ is a Bagnera-De Franchis manifold with group $G \cong \ZZ/m$, then statements (1) and (2) specialize to
	\begin{itemize}
		\item[(1')] Every line bundle $L \in \Pic(T)^G$ admits a $G$-linearization, and two linearizations on $L$ differ by an $m$-th root of unity.\\
		Moreover, if $\sL \in \Pic(X)$ pulls back to a line bundle $L \in \Pic(T)^G$ with Appell-Humbert data $(H,\rho)$, a cocycle $[f] \in H^1(\Ga, \hol_V^*)$ corresponding to $\sL$ is determined by
		$$f_\ga(z) := \rho(\la)^{1/m} \exp(\frac{\pi}{m} H(z,\la) + \frac{\pi}{2m^2} H(\la,\la)), $$
		where
		\begin{itemize}
			\item[\textbullet] $\ga$ is a lift of a generator $g$ of $G$, which acts on $V_1 \oplus V_2$ as $(z_1,z_2) \mapsto (z_1 + b_1, \al z_2 + b_2),$
			\item[\textbullet] $\la := \ga^m \in \Lam$, 
			\item[\textbullet] $\rho(\la)^{1/m}$ is an  $m$-th root of $\rho(\la)$ in $\CC$.
		\end{itemize}
		\item[(2')] The map $\psi$ is onto, in particular, we have  exact sequences
		$$0 \to H^1(X,\ZZ) \to H^1(T,\ZZ)^G \to H^2(G,\ZZ) \to 0,$$
		 $$ (**) \  0    \ra    H^1(G, (\Lam^{\vee})) \ra H^2 (X, \ZZ)   \ra   ker [ H^2 (T, \ZZ)^G \ra  H^2( G , (\Lam^{\vee}))]  \ra 0 ,$$
		 with $H^1(G, (\Lam^{\vee}))= Tors (H^2 (X, \ZZ)   )$.
		 \item[(3')] The first Chern class map $c_1$ applied to the exact sequence 
		 $$ 0 \ra H^1(G, \CC^*)   \ra \Pic(X) \ra  \Pic(T)^G \ra H^2(G, \CC^*) $$
		 sends $H^1(G, \CC^*)$ to $Pic^0(X)$, and sends $(T^*)^G = \Pic^0(T)^G $ onto $H^1(G, \Lam^{\vee})$.

	\end{itemize}
\end{theorem}

\begin{proof}
	The last part of assertion (1') is not proven yet. If $X = T/G$ is a Bagnera-De Franchis manifold and $\sL$ is a line bundle on $X$, we aim to give an element $[f] \in H^1(\Ga, \hol_V^*)$ corresponding to $\sL$. If $\la' \in \Lam \subset \Ga$, we can choose $f_{\la'}(z)$ to be in Appell-Humbert normal form,
	\begin{align*}
	f_{\la'}(z) = \rho(\la')  \ \exp(\pi H(z,\la') + \frac{\pi}{2} H(\la', \la')),
	\end{align*}
	as already noted 
	. Since $G$ is cyclic, it remains to determine an element $f_\ga(z)$, which satisfies the cocycle condition
	$$f_\la(z) = f_\ga(\ga^{m-1}z) \cdot ... \cdot f_\ga(z), \;\;\;\; \la := \ga^m = mb_1 \in \Lam \cap V_1.$$
	
	It remains to check that the definition of $f_\ga$ in the statement of the Theorem satisfies this condition. We calculate
	
	$$f_\ga(\ga^{m-1}z) \cdot ... \cdot f_\ga(z) = \rho(\la) \ \exp(\frac{\pi}{m} H((\ga^{m-1} + ... + \ga + Id)z, \la) + \frac{\pi}{2m} H(\la, \la)).$$
	
	Writing $z = (z_1,z_2)$, $z_j \in V_j$, we obtain that 
	
	$$ (\ga^{m-1} + ... + \ga + Id)z = (mz_1 + \frac{m(m-1)}{2} b_1, b_2'), \text{ for some } b_2' \in V_2. $$
	We note that, since $H$ is $G$-invariant, we obtain that $H(w_2,w_1) = 0$ for any $w_j \in V_j$. This implies, together with $\la = mb_1 \in V_1$, that
	
	$$\frac{\pi}{m} H((\ga^{m-1} + ... + \ga + Id)z, \la) = \pi H(z,\la) + \frac{\pi (m-1)}{2} H(b_1, \la),$$ 
	
	and finally the desired
	
	$$ f_\ga(\ga^{m-1}z) \cdot ... \cdot f_\ga(z) = f_\la(z).$$
	
	For the other yet unproven assertion (3'), we use the exact sequence (Pic) stating that we have a surjection 
	$(T^*)^G \ra H^1(G, \Lam^{\vee})$. In particular, $H^1(G, \CC^*)$ maps to zero in $(T^*)^G$, hence has trivial integral 
	Chern class. Indeed $H^1(G, \CC^*)$ maps to $H^2(G, \ZZ)$ which maps to zero inside $H^2(X, \ZZ)$.
	
\end{proof}

\bigskip

\section{Tangent bundles of Bagnera de Franchis manifolds and counterexamples to  the Severi conjecture}

\begin{theorem}\label{bdf}
The tangent bundle of a Bagnera de Franchis manifold $X = T/G$ is topologically trivial, in particular all its integral
Chern classes $c_i(X) =0 \in H^*(X,\ZZ)$.
\end{theorem}

\begin{proof}

 $G$ is cyclic: more generally, if $G$ is Abelian, setting  $T = V / \Lam$, the vector space $V$ splits as a direct sum of character spaces 
$$V = \oplus_{\chi \in G^*} V_{\chi},$$
 
and $\Theta_X$ is a direct sum of line bundles $L_1, \dots, L_n$ corresponding to some character $\chi_i  \in G^* = H^1(G, \CC^*)$.

It suffices then to show that these line bundles are topologically trivial, which follows
since we saw in (3') of Theorem \ref{summarize} that $c_{1, \ZZ}(L_i)=0$, and for a line bundle topological triviality is equivalent to
triviality of its integral Chern class. 

\end{proof}

The previous theorem gives a counterexample of the topological version of Severi's conjecture, but we 
can give a stronger counterexample, where all Chern classes are zero in the Chow ring
and not only in the cohomology ring.

\begin{theorem}
There are projective Bagnera de Franchis manifolds $X = T/G$, which are not complex tori,
such that all their 
Chern classes $c_i(X) $ are zero in the Chow group.
\end{theorem}

\begin{proof}
Consider a product of three elliptic curves, $ T = E_1 \times E_2 \times E_3$ and the affine action of
$G = \ZZ/2$ on $T$ such that
$$g (a_1, a_2,a_3) = (a_1 + \eta_1 , - a_2, -a_3),$$
where $\eta_1$ is a nontrivial point of order 2.

Then $\Theta_X \cong \hol_X \oplus L \oplus L$, where $L$ is the nontrivial bundle of 2-torsion
corresponding to the unique embedding $ G \ra \CC^*$.

Since as we saw $ H^1(G, \CC^*)$ maps to zero in $H^2(X, \ZZ)$, $L$ is an element of 
$Pic^0(X) \cong E_1 / \langle \eta_1\rangle$, hence $L$ pulls back from the elliptic curve $Alb(X)$,
$ L = \pi^* (\sL)$.

Then, since $ 2 c_1(L) = 0$,
 $$ c( \Theta_X) = (1 + c_1(L))^2 =   1 + \pi^* (  c_1(\sL)^2) =1$$
 in the Chow ring of $X$, since $c_1 (\sL)^2 = 0$ in the Chow wing of $Alb(X)$.

\end{proof}

We shall briefly  recall in the last section
that any compact K\"ahler manifold $X$ with $c_i(X) =0 \in H^*(X,\QQ), \ \forall i,$
is a Hyperelliptic manifold. However

\begin{theorem}
There are Hyperelliptic manifolds  $X = T/G$
such that not all their integral 
Chern classes $c_i(X) \in H^*(X,\ZZ)$ are equal to zero.
\end{theorem}
\begin{proof}
Consider the product of two elliptic curves and an Abelian surface, $ T = E_1 \times E_2 \times A_3$ and the affine action of
$G = \ZZ/2 \oplus \ZZ/2$ such that
$$g_{12} (a_1, a_2,a_3) = (-a_1 + \eta_1 , - a_2, a_3 + \eta_3),$$
$$g_{13} (a_1, a_2,a_3) = (-a_1  , a_2 + \eta_2, - a_3 + \eta_3),$$
$$g_{23} (a_1, a_2,a_3) = (a_1 + \eta_1 , - a_2 + \eta_2, - a_3 ),$$
where $\eta_1, \eta_2, \eta_3$ are  respective nontrivial torsion points of order 2.

Then $\Theta_X \cong L_1 \oplus L_2 \oplus L_3 \oplus L_3$, where the $L_i$ are nontrivial bundles of 2-torsion,
corresponding to the three nontrivial characters  $ G \ra \CC^*$.

In this case $H^1(T, \ZZ)^G =0$, hence $Pic^0(X)$ is trivial. Hence the three Chern classes
$c_1(L_i)$ are the three nontrivial elements of $ H^2(G, \ZZ) \subset H^2(X, \ZZ)$ (as $\psi =0$).

Hence $$ c( \Theta_X) = (1 + c_1(L_1)) (1 + c_1(L_2)) (1 + c_1(L_3))^2,$$
and $c_1(\Theta_X) = c_1(L_3) \neq 0 \in H^2(X, \ZZ)$.

\end{proof}

\bigskip

 The above theorem raises the interesting question
 
 \begin{question}
(1)  Is there an easy classification of  the Hyperelliptic manifolds  $X = T/G$
whose  integral 
Chern classes $c_i(X) \in H^*(X,\ZZ)$ are all equal to zero? 

(2) In particular,  of all compact K\"ahler manifolds whose tangent bundle
is topologically trivial ?

 \end{question}

\section{Recalling the Chern classes characterization of Hyperelliptic Manifolds}

A complex torus $T$ has holomorphically trivial tangent bundle and conversely a compact K\"ahler manifold $X$ with
holomorphically  trivial tangent (or cotangent) bundle
is a complex torus.

Because, if  $\Omega^1_X \cong \hol_X^n$, then $H^0(\Omega^1_X ) \cong \CC^n$,
hence the Albanese Variety has dimension $n = dim(X)$ and the Albanese map $ \al : X \ra Alb(X)$ is 
a finite unramified covering.

More generally,   \cite{amoros}, Prop. 8.1,   states that   a compact complex manifold $X$ with $\Theta_X \cong \hol_X^n$
is a complex torus if the Lie algebra generated by the holomorphic vector fields is commutative.

In particular, all the Chern classes of $T$, $c_i(T)  \in H^{2i}(T , \ZZ)$ are trivial.

 Recall now the
 \begin{rem}\label{isogeny}
 {\bf (Isogeny principle):} If we have a finite unramified map $ f : Z \ra X$, then $c_{i, \QQ}(Z) = 0 \in H^{2i}(Z, \QQ)$ if and only 
 if $c_{i, \QQ}(X) = 0 \in H^{2i}(X, \QQ)$. 
 
 Defining {\bf isogeny} between manifolds as the equivalence relation generated by the existence
 of such finite unramified maps (which we can further assume to be Galois coverings), we see that the  set of manifolds with a vanishing rational Chern class,
   $$\{ X | c_{i, \QQ}(X)=0 \}$$  consists of a union of  isogeny classes.
 
 \end{rem} 

If we take a hyperelliptic manifold $X = T / G$, then, by the isogeny principle
 the rational Chern classes are trivial
(equivalenty, $ c_i(X) = 0 \in H^{2i}(X , \RR)$).

 The new differential theoretic methods turned out in the 50's to be quite powerful for complex algebraic geometry,  for instance Apte \cite{apte}
 proved (see also \cite{kobayashi}, page 116) that a compact K\"ahler-Einstein manifold $X$
 such that $c_{1,\RR}(X)= 0,  c_{2,\RR} (X) = 0 \in H^*(X, \RR)$ is flat, that is, $X$ is a  hyperelliptic manifold. 
 
 In 1978 Yau \cite{yau} (he obtained the Fields medal for this result) showed that a compact K\"ahler manifold with 
 $ c_{1,\RR}(X) = 0 \in H^2(X, \RR)$ admits a K\"ahler-Einstein metric, that is a metric such that its Ricci form
 is identically zero. Hence the following theorem was proven:
 
 \begin{theorem}\label{zerochern} (Yau)
  A compact K\"ahler manifold $X$
 such that $c_{1,\RR}(X)= 0,  c_{2,\RR} = 0 $ in $ H^*(X, \RR)$, is a  hyperelliptic manifold. 
  \end{theorem}
  
See \cite{lu-bt} for the latest touch concerning generalizations of the theorem to the case where $X$ is singular.

  \bigskip

{\bf Acknowledgements:} Thanks to   Francesco Baldassarri for bringing the paper \cite{baldassarri} by Mario Baldassarri to our attention.

Section 2 on the Picard groups of Hyperelliptic Manifolds benefited greatly  from conversations with  my former student Andreas Demleitner, to whom  I owe part  (1') of Theorem \ref{summarize}.

Thanks to a referee  (for our previous JAMS submission) for spotting an inaccuracy in  a previous version of the article.

\end{document}